\documentclass{amsart}

\usepackage{mathpazo}
\usepackage{amsmath,amssymb,amscd,color}
\usepackage{mathrsfs}
\usepackage{float}
\usepackage{tikz-cd}
\usetikzlibrary{positioning}
\usetikzlibrary{decorations.markings}
\usetikzlibrary{arrows, scopes}
\tikzcdset{arrow style=tikz, diagrams={>=stealth}, every label/.append style = {font = \normalsize}}
\usepackage{comment}
\definecolor{green}{RGB}{0,144,0}
\definecolor{bluegreen}{RGB}{17,100,180}
\usepackage[linkcolor = bluegreen, citecolor = bluegreen, colorlinks = True]{hyperref}
\numberwithin{equation}{section}
\numberwithin{figure}{section}
\numberwithin{table}{section}
\usepackage[all]{hypcap}
\usepackage{mathtools}

\tikzcdset{arrows={line width=1pt, line cap = round}}

\newtheorem{theorem}{Theorem}[section]
\newtheorem*{theorem*}{Theorem}
\newtheorem{conjecture}[theorem]{Conjecture}

\newtheorem{proposition}[theorem]{Proposition}

\newtheorem{question}[theorem]{Question}

\newcommand{\calH}{\mathcal{H}}
\newcommand{\T}{\mathbb{T}}
\newcommand{\C}{\mathbb{C}}
\newcommand{\Z}{\mathbb{Z}}
\newcommand{\R}{\mathbb{R}}
\DeclareMathOperator{\SL}{\text{SL}}
\DeclareMathOperator{\Sym}{\text{Sym}}
\DeclareMathOperator{\Alt}{\text{Alt}}

\title{Non-planarity of $\SL(2,\Z)$-orbits of origamis in $\mathcal{H}(2)$}

\author{Luke Jeffreys}
\address{School of Mathematics , University of Bristol, Fry Building, Woodland Road, Bristol BS8 1UG}
\curraddr{}
\email{luke.jeffreys@bristol.ac.uk}
\thanks{The first author is thankful for support from the Heilbronn Institute for Mathematical Research}

\author{Carlos Matheus}
\address{Centre de Math{\'e}matiques Laurent Schwartz, {\'E}cole Polytechnique, 91128 Palaiseau Cedex, France}
\curraddr{}
\email{carlos.matheus@math.cnrs.fr}

\date{}

\subjclass[2020]{Primary: 32G15, 30F30, 30F60. Secondary: 05C10.}

\begin{document}

\begin{abstract}
We consider the $\SL(2,\Z)$-orbits of primitive $n$-squared origamis in the stratum $\calH(2)$. In particular, we consider the 4-valent graphs obtained from the action of $\SL(2,\Z)$ with respect to a generating set of size two. We prove that, apart from the orbit for $n = 3$ and one of the orbits for $n = 5$, all of the obtained graphs are non-planar. Specifically, in each of the graphs we exhibit a $K_{3,3}$ minor, where $K_{3,3}$ is the complete bipartite graph on two sets of three vertices.
\end{abstract}

\maketitle


\section{Introduction}

An origami is an orientable connected surface obtained from a finite collection of unit squares in $\R^{2}$ by identifying by translation left-hand sides with right-hand sides and top sides with bottom sides. Equivalently, an origami is a connected Riemann surface $X$ equipped with a holomorphic one-form $\omega$ such that $X$ is obtained as a finite cover $\pi:X\to \T^{2} := \C/(\Z+i\Z)$ branched only at the origin $\bf{0}\in \T^{2}$ and $\omega = \pi^{*}(dz)$.

Origamis are a special case of a more general object called a translation surface -- a connected Riemann surface $X$ equipped with a non-zero holomorphic one-form $\omega$. The moduli space of all translation surfaces of a fixed genus $g$ is denoted by $\calH_{g}$. By the Riemann-Roch Theorem, the orders of the zeros of a holomorphic one-form on a genus $g$ Riemann surface must sum to $2g-2$. As such, the moduli space $\calH_{g}$ is naturally stratified as
\[\calH_{g} = \bigsqcup_{\substack{(k_{1},\ldots,k_{s}),\,\,\, k_{i}\geq 1  \\ \sum k_{i} = 2g-2}} \calH(k_{1},\ldots,k_{s}),\]
where the stratum $\calH(k_{1},\ldots,k_{s})$ is the subset of $\calH_{g}$ consisting of those translation surfaces $(X,\omega)$ where $\omega$ has $s$ zeros of orders $k_{1},\ldots,k_{s}$. Each stratum is a complex orbifold of dimension $2g + s - 1$, and in a natural coordinate system on these strata origamis can be thought of as the integer lattice points. This point of view was utilised by Eskin-Okounkov~\cite{EO} and Zorich~\cite{Z1} to calculate the volumes of strata with respect to a natural measure. For more details on translation surfaces and their applications, we direct the reader to the surveys of Forni and the second author~\cite{FM}, Yoccoz~\cite{Y}, and Zorich~\cite{Z2}.

Each stratum of translation surfaces admits a natural action of $\SL(2,\R)$ which restricts to an action of $\SL(2,\Z)$ on the origamis in that stratum. In this paper, we consider the $\SL(2,\Z)$-orbits of primitive $n$-squared origamis in the stratum $\calH(2)$. These orbits were classified in the works of McMullen~\cite{M} and Hubert-Leli{\`e}vre~\cite{HL}. We note that an origami in $\calH(2)$ requires at least 3 squares. For $n = 3$ or $n\geq 4$ even there is a single $\SL(2,\Z)$-orbit of primitive $n$-squared origamis in $\calH(2)$. For $n \geq 5$ odd, there are two $\SL(2,\Z)$-orbits of primitive $n$-squared origamis in $\calH(2)$. These two orbits, defined more explicitly in Subsection~\ref{s:orbits}, are called the A- and B-orbits, respectively. By considering $\SL(2,\Z) = \langle T, S\rangle$, where
\[T = \begin{bmatrix}
1 & 1 \\
0 & 1
\end{bmatrix}
\,\,\,\text{and}\,\,\,
S = \begin{bmatrix}
1 & 0 \\
1 & 1
\end{bmatrix},\]
each $\SL(2,\Z)$-orbit can be realised as a finite 4-valent graph which we will denote by $\mathcal{G}_{n}$ for $n = 3$ or $n\geq 4$ even, and for $n\geq 5$ odd by $\mathcal{G}_{n}^{A}$, or $\mathcal{G}_{n}^{B}$ for the A- and B-orbits, respectively.

It is a conjecture of McMullen that this family of graphs forms a family of expanders. Here we prove that, apart from $\mathcal{G}_{3}$ and $\mathcal{G}_{5}^{B}$, these graphs are non-planar. That is, our main result is the following.

\begin{theorem}\label{t:main}
The graphs $\mathcal{G}_{n}$, $\mathcal{G}_{n}^{A}$, and $\mathcal{G}_{n}^{B}$ are all non-planar with the exception of $\mathcal{G}_{3}$ and $\mathcal{G}_{5}^{B}$.
\end{theorem}

This provides indirect evidence for McMullen's conjecture. Indeed, it follows from the separator theorem for planar graphs of Lipton-Tarjan~\cite{LT} that planar graphs cannot form an expander family. We direct the reader to the exposition of this fact in de Courcy-Ireland's recent work~\cite[Section 5]{dCI} demonstrating that, for certain primes $p$, Markoff graphs modulo $p$ -- another conjectured expander family -- are non-planar.

Our proof relies on the theorem of Wagner~\cite{W} and Kuratowski~\cite{K} characterising planar graphs in terms of forbidden minors. Recall that a graph $H$ is realised as a minor inside a graph $G$ if one can perform a sequence of edge-contractions, edge-deletions, and deletions of isolated vertices in order to transform the graph $G$ into the graph $H$. Wagner and Kuratowski's characterisation can then be worded as follows: a graph $G$ is planar if and only if it does not contain a $K_5$ or $K_{3,3}$ minor, where $K_{5}$ is the complete graph on five vertices and $K_{3,3}$ is the complete bipartite graph on two sets of three vertices. We prove the non-planarity claim of Theorem~\ref{t:main} by realising a $K_{3,3}$ minor in each case.

The planar graphs $\mathcal{G}_{3}$ and $\mathcal{G}_{5}^{B}$ are shown in Figures~\ref{f:G3} and~\ref{f:G5B} at the end of Section~\ref{s:minors}.

\subsection*{Weierstrass curves and an alternate generating set}

A Teichm{\"u}ller curve is an algebraic and isometric immersion of a finite-volume hyperbolic Riemann surface into the moduli space $\mathcal{M}_{g}$ of Riemann surfaces of genus $g$. McMullen~\cite{M} classified all of the Teichm{\"u}ller curves in genus two. The main source of such Teichm{\"u}ller curves are the so called Weierstrass curves $W_{D}$ parameterised by integers $D\geq 5$ with $D\equiv 0$ or $1$ modulo $4$. These curves consist of those Riemann surfaces $X\in\mathcal{M}_{2}$ whose Jacobians admit real multiplication by the quadratic order $\mathcal{O}_{D} := \Z[x]/\langle x^{2}+bx+c\rangle$, $b,c\in\Z$ with $D = b^{2}-c$, and for which there exists a holomorphic one-form $\omega$ on $X$ such that $(X,\omega)\in\calH(2)$ and $\mathcal{O}_{D}\cdot\omega\subset\C\cdot\omega$. A translation surface $(X,\omega)\in\calH(2)$ projects to $W_{n^{2}}$ if and only if $(X,\omega)$ is an $n$-squared origami. The classification of the $\SL(2,\Z)$-orbits of primitive origamis in $\calH(2)$ is equivalent to the fact that $W_{n^{2}}$ is connected for $n = 3$ and $n\geq 4$ even, and has two connected components for $n\geq 5$ odd (see the classification by McMullen~\cite{M}).

It follows from work of Mukamel~\cite[Corollary 1.4]{Muk} that all of the components of $W_{n^{2}}$ have genus zero for $n\leq 7$ and one of the components has genus zero for $n = 9$. We see that this planarity range does not agree with the planarity of our graphs $\mathcal{G}_{n}$. However, if one considers the generating set $\{P,R\}$ for $\SL(2,\Z)$ where 
\[P = \begin{bmatrix}
0 & 1 \\
-1 & 1
\end{bmatrix}
\,\,\,\text{and}\,\,\,
R = \begin{bmatrix}
0 & -1 \\
1 & 0
\end{bmatrix},\]
and forms the associated orbit graphs in this setting, which we shall denote by $\overline{\mathcal{G}}_{n}$, $\overline{\mathcal{G}}_{n}^{A}$, and $\overline{\mathcal{G}}_{n}^{B}$, then computational experiments lead us to make the following conjecture.

\begin{conjecture}
The graphs $\overline{\mathcal{G}}_{n}$, $\overline{\mathcal{G}}_{n}^{A}$, and $\overline{\mathcal{G}}_{n}^{B}$ are planar for $n\leq 7$. The graph $\overline{\mathcal{G}}_{9}^{B}$ is planar. All remaining graphs are non-planar.
\end{conjecture}

In other words, we conjecture that planarity in this setting exactly agrees with the range for genus zero components of $W_{n^{2}}$ given by Mukamel. The planarity for $n\leq 7$ and of $\overline{\mathcal{G}}_{9}^{B}$ is computationally confirmed and so the conjecture is that the remainder are all non-planar. The non-planarity of $\overline{\mathcal{G}}_{n}$ has been confirmed for even $n$ in the range $8\leq n\leq 16$, of $\overline{\mathcal{G}}_{n}^{A}$ in the range $9\leq n\leq 17$, and of $\overline{\mathcal{G}}_{n}^{B}$ in the range $11\leq n\leq 23$. 

Finally, we also remark that there exists a generalisation of the separator theorem of Lipton-Tarjan given by Gilbert-Hutchinson-Tarjan~\cite{GHT} from which it follows that a family of expanders must have genus tending to infinity. So, for us, if the family of graphs $\mathcal{G}_{n}$, or $\overline{\mathcal{G}}_{n}$, are a family of expanders then the genus of these graphs must tend to infinity. In the Weierstrass curve setting, Mukamel~\cite[Corollary 1.3]{Muk}, building on work of Bainbridge~\cite{B}, has shown that the genus of any component of $W_{D}$ tends to infinity as $D$ tends to infinity.

\subsection*{Acknowledgements}

We thank Curtis McMullen for bringing additional references to our attention and for suggesting the investigation of the alternate generating set for $\SL(2,\Z)$. We also thank the anonymous referee for their careful reading of this article.


\section{Origami preliminaries} 

Here we give an introduction to origamis and the classification of their $\SL(2,\Z)$-orbits in the stratum $\calH(2)$. In particular, we will discuss their algebraic description using pairs of permutations which we will make use of in the remainder of the paper.

\subsection{Origamis}

An origami is an orientable connected surface obtained from a collection of unit squares in $\R^{2}$ by identifying by translation left-hand sides with right-hand sides and top sides with bottom sides. See for example the surface in Figure~\ref{f:example} where sides with the same label are identified by translation.

\begin{figure}[t]
\begin{center}
\begin{tikzpicture}[scale = 1.25]
\draw [line width = 0.55mm, line cap = round] (0,0)--node{$\geq$}(0,1);
\draw [line width = 0.55mm, line cap = round] (0,1)--node{$|$}(1,1);
\draw [line width = 0.55mm, line cap = round] (1,1)--node{$||$}(2,1);
\draw [line width = 0.55mm, line cap = round] (2,1)--node{$|||$}(3,1);
\draw [line width = 0.55mm, line cap = round] (3,1)--node{$\geq$}(3,0);
\draw [line width = 0.55mm, line cap = round] (0,0)--node{$|||$}(1,0);
\draw [line width = 0.55mm, line cap = round] (1,0)--node{$||$}(2,0);
\draw [line width = 0.55mm, line cap = round] (2,0)--node{$|$}(3,0);
\node at (0.5,0.5) {1};
\node at (1.5,0.5) {2};
\node at (2.5,0.5) {3};
\draw [densely dashed, line width = 0.35mm, line cap = round] (1,0)--(1,1);
\draw [densely dashed, line width = 0.35mm, line cap = round] (2,0)--(2,1);
\end{tikzpicture}
\end{center}
\caption{An origami in $\calH(2)$}\label{f:example}
\end{figure}

An $n$-squared origami can also be described by two permutations $h$ and $v$ in the symmetric group $\Sym(n)$. These permutations are obtained as follows. Firstly, we number the squares from 1 to $n$.  We then define $h$ to be the element of $\Sym(n)$ such that $h(i) = j$ if and only if the right-hand side of the square labelled by $i$ is identified with the left-hand side of the square labelled by $j$. Similarly, we define $v$ to be the permutation satisfying $v(i) = j$ if and only if the top side of the square labelled by $i$ is glued to the bottom side of the square labelled by $j$. For example, the origami in Figure~\ref{f:example} can be described by the pair $(h,v) = ((1,2,3),(1,3))$. Since a different labelling of the squares could produce a different pair of permutations, an origami actually corresponds to a pair $(h,v)$ considered up to simultaneous conjugation of $h$ and $v$. However, we will abuse notation in this article and denote an origami simply by the pair $(h,v)$. When we need to make it clear, we will use the notation $(h,v)\simeq (h',v')$ if the pairs are equivalent by simultaneous conjugation.

We will not describe the constructions of strata of translations surfaces in this paper. Instead, we direct the reader to the expositions found in the surveys discussed in the introduction. For our purposes, it will suffice to say that an origami $O = (h,v)$ lies in the stratum $\calH(k_{1},\ldots,k_{s})$ if when writing the commutator $[h,v] = hvh^{-1}v^{-1}$ as a product of disjoint non-trivial cycles we obtain $s$ cycles of lengths $k_{1}+1,\ldots,k_{s}+1$. For example, the origami in Figure~\ref{f:example} has $[h,v] = (1,3,2)$ consisting of a single 3-cycle and lies in the stratum $\calH(2)$. Similarly, the origami $O' = ((2,3,4),(1,2))$ has commutator $(1,2,3)$ and hence also lies in $\calH(2)$.

We will also make use of the monodromy group $\sigma(O)$ of an origami $O$ which is defined to be the subgroup of $\Sym(n)$ generated by $h$ and $v$; that is, $\sigma((h,v)) = \langle h, v\rangle$. An origami is then said to be primitive if its monodromy group is a primitive subgroup of $\Sym(n)$. Topologically, an origami is said to be primitive if it is not a proper cover of another origami (other than the unit-square torus of which all origamis are a cover). Both of these notions of primitivity are equivalent to one another. For more on the translation between the algebraic and topological descriptions of origamis we direct the reader to the thesis of Zmiaikou~\cite{Zm}.

We remark that the definitions of the monodromy group and of the stratum of an origami are well-defined in the sense that they are invariant under simultaneous conjugation of $h$ and $v$.

\subsection{Orbit classification} \label{s:orbits}

We now discuss the action of $\SL(2,\Z)$ on origamis and the classification of orbits in the stratum $\calH(2)$.

The group $\SL(2,\Z)$ acts on origamis by its natural action on the plane. Indeed, up to cutting and pasting, unit squares are mapped to unit squares and parallel sides are sent to parallel sides. Consider $\SL(2,\Z) = \langle T, S\rangle$, where
\[T = \begin{bmatrix}
1 & 1 \\
0 & 1
\end{bmatrix}
\,\,\,\text{and}\,\,\,
S = \begin{bmatrix}
1 & 0 \\
1 & 1
\end{bmatrix}.\]

The matrix $T$ acts by horizontally shearing the origami to the right while the matrix $S$ acts by vertically shearing the origami upwards.

It can then be checked, see Figure~\ref{f:T-action}, that
\[T((h,v)) = (h,vh^{-1}).\]
Similarly, it can be argued by symmetry (exchanging the roles of $h$ and $v$) that 
\[S((h,v)) = (hv^{-1},v).\]
In particular, it follows that the number of squares, the stratum, the primitivity, and the monodromy group of an origami are invariant under the action of $\SL(2,\Z)$. As such, it makes sense to discuss the $\SL(2,\Z)$-orbits of primitive $n$-squared origamis in a given stratum.

\begin{figure}[t]
\begin{center}
\includegraphics[scale=1]{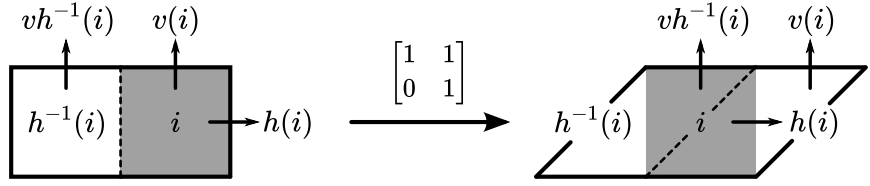}
\end{center}
\caption{The action of $T$ on the permutations $h$ and $v$ of an origami.}\label{f:T-action}
\end{figure}

It is a consequence of the classification of $\SL(2,\Z)$-orbits of primitive origamis in $\calH(2)$ due to McMullen~\cite{M} and Hubert-Leli{\`e}vre~\cite{HL} that the monodromy group is in fact a strong $\SL(2,\Z)$-invariant. That is, we have the following statement.

\begin{theorem}[McMullen~{\cite{M}}, Hubert-Leli\`evre~{\cite{HL}}]
Let $n = 3$ or let $n\geq 4$ be even. Then there is a single $\SL(2,\Z)$-orbit of primitive $n$-squared origamis in $\calH(2)$. Every origami in this orbit has monodromy group the full symmetric group $\Sym(n)$.

Let $n\geq 5$ be odd. Then there are two $\SL(2,\Z)$-orbits, the A-orbit and the B-orbit, of primitive $n$-squared origamis in $\calH(2)$. Every origami in the A-orbit has monodromy group the full symmetric group $\Sym(n)$, while every origami in the B-orbit has monodromy group the alternating group $\Alt(n)$.
\end{theorem}


\section{Construction of the minor}\label{s:minors}

In this section, we construct the $K_{3,3}$ minors in each of the orbits.

\subsection{Even squared orbits}

Let $n\geq 4$ be even and define the origamis $O_{1},\ldots,O_{6}$ as follows:
\begin{align*}
O_{1} &= ((n-1,n), (1,2,\ldots,n-1)), \\
O_{2} &= ((1,2,\ldots,n),(1,n-1,n-3,\ldots,3,n,n-2,\ldots,2)), \\
O_{3} &= ((1,2,\ldots,n),(n-1,n)), \\
O_{4} &= ((1,2,\ldots,n),(2,3,\ldots,n)), \\
O_{5} &= ((1,2,\ldots,n),(2,n,n-1,\ldots,3)),\,\,\,\text{ and } \\
O_{6} &= ((2,3,\ldots,n),(1,2,n,n-1,\ldots,3)).
\end{align*}
By calculating $[h,v]$ in each of the above cases, it is easy to check that these origamis lie in $\calH(2)$. We then have the following.

\begin{proposition}\label{p:even}
Let $n\geq 4$ be even and define $O_{1},\ldots,O_{6}$ as above. Then there exists a $K_{3,3}$ minor in $\mathcal{G}_{n}$ with bipartition $\{\{O_{1},O_{2},O_{3}\},\{O_{4},O_{5},O_{6}\}\}$.
\end{proposition}

\begin{proof}
Let $\overline{ij}$ denote the path in the $K_{3,3}$ minor between $O_{i}$ and $O_{j}$. We claim that the following paths realise the minor:
\begin{align*}
\overline{14} &= S^{-1}, \\
\overline{15} &= S, \\
\overline{16} &= S\circ T, \\
\overline{24} &= T^{n-3}, \\
\overline{25} &= T^{-1}, \\
\overline{26} &= T^{-(n-4)}\circ S^{-1}, \\
\overline{34} &= T^{-1}, \\
\overline{35} &= T,\,\,\,\text{ and } \\ 
\overline{36} &= T\circ S.
\end{align*}

We need only check that these paths do connect the claimed origamis and that they are pairwise disjoint.

It is readily checked that we have
\[\overline{14} = \{O_{1},O_{4}\},\,\, \overline{15} = \{O_{1},O_{5}\},\,\, \overline{25} = \{O_{2},O_{5}\},\,\, \overline{34} = \{O_{3},O_{4}\},\]
and
\[\overline{35} = \{O_{3},O_{5}\}.\]
Similarly, we have
\[\overline{16} = \{O_{1},O',O_{6}\},\]
where
\[O' = ((n-1,n),(1,2,\ldots,n)),\]
and
\[\overline{36} = \{O_{3},O'',O_{6}\},\]
where
\[O'' = ((1,2,\ldots,n-1),(n-1,n)) \simeq ((2,3,\ldots,n),(1,2)).\]

Finally, we must check the paths $\overline{24}$ and $\overline{26}$. We observe that the origami $O_{4}$ has a $T$-orbit of length $n$ which begins with $O_{4},O_{3},O_{5}$, and $O_{2}$. The remaining $n-4$ points of the orbit form the path $\overline{24}$ and, since the conjugacy class of the first permutation remains unchanged along this orbit, we see that this path is disjoint from all of the paths considered so far.

For $\overline{26}$, we observe that
\[S^{-1}(O_{2}) = ((1,n,n-1,\ldots,3),(1,n-1,n-3,\ldots,3,n,n-2,\ldots,2))\]
\[\simeq ((2,3,\ldots,n),(1,2,4,\ldots,n,3,5,\ldots,n-1)).\]
Observe that for $n=4$ we have $S^{-1}(O_{2}) = O_{6}$. We will assume $n\geq 6$ in the remainder. It can then be checked that $S^{-1}(O_{2})$ lies in the same $T$-orbit as $O_{6}$. However, $O'' = T^{-1}(O_{6}) = T^{2}\circ S^{-1}(O_{2})$ and so, to avoid $O''$, we must complete the path $\overline{26}$ using powers of $T^{-1}$. The $T$-orbit of $O_{6}$ is of size $n-1$ and so we require $T^{-(n-4)}$ to take $S^{-1}(O_{2})$ to $O_{6}$. As $O''$ is the only origami in the $T$-orbit of $O_{6}$ that we have used already (apart from $O_{6}$, of course) we see that this path is disjoint from all of those we have considered above. This completes the proof.
\end{proof}

\subsection{A-orbits}

Let $n\geq 5$ be odd and define the origamis $O_{1},\ldots,O_{6}$ as follows:
\begin{align*}
O_{1} &= ((n-1,n), (1,2,\ldots,n-1)), \\
O_{2} &= ((1,2,\ldots,n),(1,n-1,n-3,\ldots,2)(n,n-2,\ldots,3)), \\
O_{3} &= ((1,2,\ldots,n),(n-1,n)), \\
O_{4} &= ((1,2,\ldots,n),(2,3,\ldots,n)), \\
O_{5} &= ((1,2,\ldots,n),(2,n,n-1,\ldots,3)),\,\,\,\text{ and } \\
O_{6} &= ((2,3,\ldots,n),(1,2,n,n-1,\ldots,3)).
\end{align*}
By calculating $[h,v]$ in each of the above cases, it is easy to check that these origamis lie in $\calH(2)$. Furthermore, they all have their monodromy group being the symmetric group and so they lie in the A-orbits. We then have the following.

\begin{proposition}
Let $n\geq 5$ be odd and define $O_{1},\ldots,O_{6}$ as above. Then there exists a $K_{3,3}$ minor in $\mathcal{G}_{n}^{A}$ with bipartition $\{\{O_{1},O_{2},O_{3}\},\{O_{4},O_{5},O_{6}\}\}$, where the paths are given exactly as in the proof of Proposition~\ref{p:even}.
\end{proposition}

\begin{proof}
The argument of Proposition~\ref{p:even} follows exactly subject to the modification of $O_{2}$. The only other change required is that now
\[S^{-1}(O_{2}) = ((2,3,\ldots,n),(1,2,4,\ldots,n-1)(3,5,\ldots,n)).\]
\end{proof}

\subsection{B-orbits}

Let $n\geq 7$ be odd and define the origamis $O_{1},\ldots,O_{6}$ as follows:
\begin{align*}
O_{1} &= ((n-2,n-1,n), (1,2,\ldots,n-2)), \\
O_{2} &= ((n-2,n-1,n), (1,2,\ldots,n)),\\
O_{3} &= ((3,4,\ldots,n),(1,2,3)), \\
O_{4} &= ((n-2,n-1,n),(1,2,\ldots,n-2,n,n-1)), \\
O_{5} &= ((3,4,\ldots,n),(1,2,\ldots,n)),\,\,\,\text{ and } \\
O_{6} &= ((3,4,\ldots,n),(1,2,3,n,n-1,\ldots,4)).
\end{align*}
Again, it is easy to check that these origamis lie in $\calH(2)$. Furthermore, they all have their monodromy group being the alternating group and so they lie in the B-orbits. We then have the following.

\begin{proposition}
Let $n\geq 7$ be odd and define $O_{1},\ldots,O_{6}$ as above. Then there exists a $K_{3,3}$ minor in $\mathcal{G}_{n}^{B}$ with bipartition
\[\{\{O_{1},O_{2},O_{3}\},\{O_{4},O_{5},O_{6}\}\}.\]
\end{proposition}

\begin{proof}
Let $\overline{ij}$ denote the path in the $K_{3,3}$ minor between $O_{i}$ and $O_{j}$. We claim that the following paths realise the minor:
\begin{align*}
\overline{14} &= T, \\
\overline{15} &= S\circ T^{-1}\circ S^{2}, \\
\overline{16} &= S^{-1}\circ T\circ S^{-2}, \\
\overline{24} &= T^{-1}, \\
\overline{25} &= T\circ S^{-1}\circ T^{-(n-3)}\circ S^{-1}, \\
\overline{26} &= S, \\
\overline{34} &= S^{-1}\circ T\circ S, \\
\overline{35} &= T^{-1},\,\,\,\text{ and } \\ 
\overline{36} &= T.
\end{align*}

We need only check that these paths do connect the claimed origamis and that they are pairwise disjoint.

It is readily checked that we have
\[\overline{14} = \{O_{1},O_{4}\},\,\, \overline{24} = \{O_{2},O_{4}\},\,\, \overline{26} = \{O_{2},O_{6}\},\,\, \overline{35} = \{O_{3},O_{5}\},\]
and
\[\overline{36} = \{O_{3},O_{6}\}.\]
Similarly, we have
\[\overline{15} = \{O_{1},O',O'',O''',O_{5}\},\]
where
\[O' = ((1,2,\ldots,n),(3,n,n-1,\ldots,4)),\]
\[O'' = ((1,2,\ldots,n),(3,n-\frac{n-3}{2},n,n-1-\frac{n-3}{2},n-1,\ldots,4,4+\frac{n-3}{2}),\]
and
\[O''' = ((1,2,\ldots,n),(1,2,2+\frac{n-1}{2},3,3+\frac{n-1}{2},\ldots,n-\frac{n-1}{2},n)),\]
\[\overline{16} = \{O_{1},O^{IV},O^{V},O^{VI},O_{6}\},\]
where
\[O^{IV} = ((1,2,\ldots,n),(3,4,\ldots,n)),\]
\[O^{V} = ((1,2,\ldots,n),(3,4+\frac{n-3}{2},4,5+\frac{n-3}{2},5,\ldots,n,n-\frac{n-3}{2})),\]
and
\[O^{VI} = ((1,2,\ldots,n),(1,\frac{n+3}{2},n,n-\frac{n-1}{2},n-1,n-\frac{n+1}{2},\ldots,3+\frac{n-1}{2},3,2)),\]
and
\[\overline{34} = \{O_{3},O^{VII},O^{VIII},O_{4}\},\]
where
\[O^{VII} = ((1,2,\ldots,n),(n-2,n,n-1)),\]
and
\[O^{VIII} = ((1,2,\ldots,n),(1,n-1,n,n-2,n-3,\ldots,2)).\]

Finally, we must check the path $\overline{25}$. We have
\[S^{-1}(O_{2}) = ((1,2,\ldots,n),(1,2,\ldots,n-3,n-1,n-2,n)),\]
\[T^{-1}(O_{5}) = ((3,4,\ldots,n),(1,2,3,5,\ldots,n,4,6,\ldots,n-1)),\]
and
\[S\circ T^{-1}(O_{5}) = ((1,2,\ldots,n),(1,n-1,n-3,\ldots,4,n,n-2,\ldots,3,2)).\]
It can be checked that $S^{-1}(O_{2})$ and $S\circ T^{-1} (O_{5})$ lie in the same $T$-orbit. However, $O' = T^{-1}\circ S\circ T^{-1} (O_{5}) = T^{2}\circ S^{-1}(O_{2})$. As such, we must complete the path using powers of $T^{-1}$. The $T$-orbit of $S^{-1}(O_{2})$ is of length $n$ and so we require $T^{-(n-3)}$ to take $S^{-1}(O_{2})$ to $S\circ T^{-1} (O_{5})$. It can also be checked that no more of the origamis we have already considered are contained in this $T$-orbit and so this path is disjoint from those we have constructed above, and so we are done.
\end{proof}

The graphs in Figures~\ref{f:G3} and~\ref{f:G5B} demonstrating the planarity of $\mathcal{G}_{3}$ and $\mathcal{G}_{5}^{B}$ complete the work of this section and the proof of Theorem~\ref{t:main}. Note that we have used directed edges in the figures so that the reader can more easily check the validity of the graphs.

\begin{figure}[H]
\begin{center}
\includegraphics[scale=1]{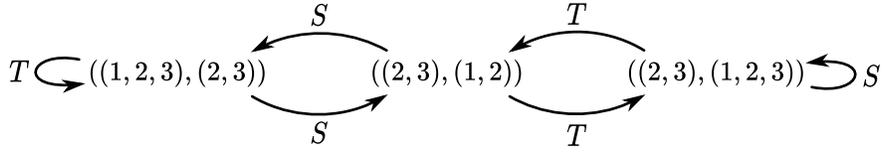}
\end{center}
\caption{The graph $\mathcal{G}_{3}$ is the undirected version of this graph.}
\label{f:G3}
\end{figure}

\begin{figure}[H]
\begin{center}
\includegraphics[scale=1]{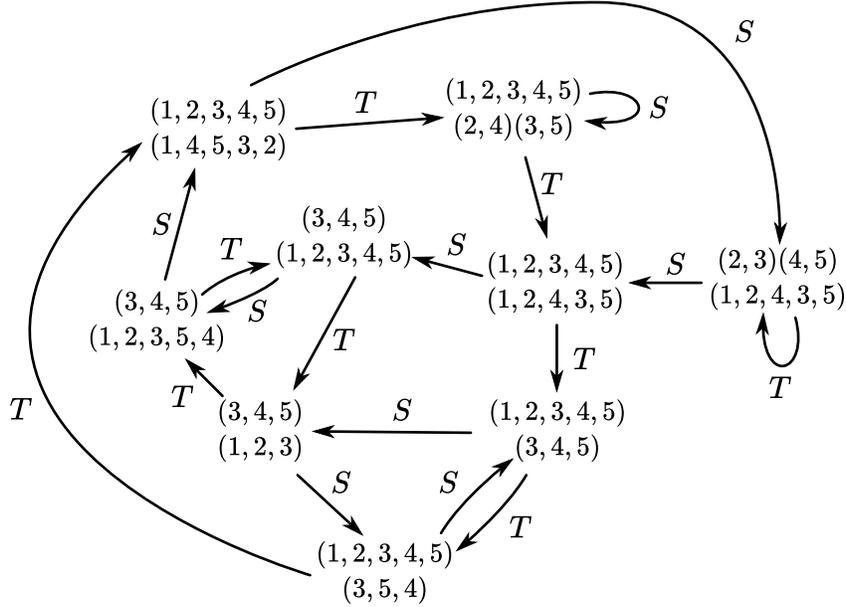}
\end{center}
\caption{The graph $\mathcal{G}_{5}^{B}$ is the undirected version of this graph.}
\label{f:G5B}
\end{figure}

\section{Further questions}

We finish with two natural questions. Firstly:

\begin{question}
Can one find generalisations of these structures that exist in the $\SL(2,\Z)$-orbits of primitive origamis in different strata?
\end{question}
We remark that the classification of these $\SL(2,\Z)$-orbits is open in general.

Secondly, the work of de Courcy-Ireland for Markoff graphs modulo $p$ gave, for certain primes $p$, a construction of a $K_{3,3}$ minor whose path lengths did not depend on the prime $p$. They called such a construction a `local' construction. In our case, this would correspond to finding a $K_{3,3}$ minor in $\mathcal{G}_{n}$ whose paths have lengths that do not depend on $n$. As such, we ask the following.

\begin{question}
Does there exist a `local' construction of a $K_{3,3}$ (or $K_{5}$) minor in the $\SL(2,\Z)$-orbits of primitive origamis in $\calH(2)$?
\end{question}



\begin{thebibliography}{99}
\bibitem{B} Bainbridge, M. Euler characteristics of Teichm{\"u}ller curves in genus two. {\em Geom. Topol.}
{\bf 11} (2007), 1887--2073

\bibitem{dCI} de Courcy-Ireland, M. Non-planarity of Markoff graphs mod p. Preprint, 2021. arXiv:2105.12411

\bibitem{EO} Eskin, Alex; Okounkov, Andrei. Asymptotics of numbers of branched coverings of a torus and volumes of moduli spaces of holomorphic differentials. {\em Invent. Math.} {\bf 145} (2001), no. 1, 59--103.

\bibitem{FM} Forni, Giovanni; Matheus, Carlos. Introduction to Teichm{\"u}ller theory and its applications to dynamics of interval exchange transformations, flows on surfaces and billiards. {\em J. Mod. Dyn.} {\bf 8} (2014), no. 3--4, 271--436.

\bibitem{GHT} Gilbert, John R.; Hutchinson, Joan P.; Tarjan, Robert Endre. A separator theorem for graphs of bounded genus. {\em Journal of Algorithms} {\bf 5} (1984), no. 3, 391--407

\bibitem{HL} Hubert, P.; Leli{\`e}vre, S. Prime arithmetic Teichm{\"u}ller discs in $\mathcal{H}(2)$, {\em Israel J. Math.} {\bf 151} (2006), 281--321.

\bibitem{K} Kuratowski, C. Sur le probl{\`e}me des courbes gauches en Topologie, {\em Fund. Math.}, {\bf 15} (1930): 271--283.

\bibitem{LT} Lipton, R. J.; Tarjan, R. E.  A separator theorem for planar graphs, {\em SIAM Journal on Applied Mathematics}, {\bf 36} (2) (1979), 177 -- 189

\bibitem{M} McMullen, C. Teichm{\"u}ller curves in genus two: discriminant and spin, {\em Math. Ann.} {\bf 333} (2005), no. 1, 87--130.

\bibitem{Muk} Mukamel, Ronen E. Orbifold points on Teichm{\"u}ller curves and Jacobians with complex multiplication. {\em Geom. Topol.} {\bf 18} (2014), 779--829

\bibitem{W} Wagner, K. Uber eine Eigenschaft der ebenen Komplexe, {\em Math. Ann.}, {\bf 114} (1937), 570--590

\bibitem{Y} Yoccoz, J.-C. Interval exchange maps and translation surfaces. {\em Homogeneous flows, moduli spaces and arithmetic}, Clay Math. Proc., {\bf 10} (2010), 1--69, Amer. Math. Soc., Providence, RI.

\bibitem{Zm} Zmiaikou, David. Origamis and permutation groups. Ph.D. thesis (2011) available at \url{http://www.zmiaikou.com/research}

\bibitem{Z1} Zorich, Anton. Square tiled surfaces and Teichm{\"u}ller volumes of the moduli spaces of abelian differentials. In {\em Rigidity in dynamics and geometry (Cambridge, 2000)}, 459–-471, Springer, Berlin, 2002.

\bibitem{Z2} Zorich, Anton. Flat Surfaces. In {\em Frontiers in Number Theory, Physics, and Geometry, I.}, 437--583, Springer, Berlin, 2006.
\end{thebibliography}
\end{document}